\documentclass[12pt]{amsart}

\usepackage{amsfonts}
\usepackage{amssymb}
\usepackage{amsmath}
\usepackage{amscd}
\usepackage{amstext}
\usepackage{ifthen}
\usepackage[all]{xy}
\usepackage{enumerate}

\makeatletter
\def\@cite#1#2{{\m@th\upshape\bfseries%
[{#1\if@tempswa{\m@th\upshape\mdseries, #2}\fi}]}}
\makeatother

\newtheorem{theorem}{Theorem}[section]
\newtheorem{lemma}[theorem]{Lemma}
\newtheorem{corollary}[theorem]{Corollary}
\newtheorem{proposition}[theorem]{Proposition}

\theoremstyle{definition}
\newtheorem{definition}[theorem]{Definition}
\newtheorem{remark}[theorem]{Remark}
\newtheorem{example}[theorem]{Example}

\numberwithin{equation}{section}

\newcommand{\bF}{{\mathbb{F}}}
\newcommand{\bN}{{\mathbb{N}}}

\newcommand{\bT}{{\mathbb{T}}}
\newcommand{\bZ}{{\mathbb{Z}}}

  \newcommand{\A}{{\mathcal{A}}}
  \newcommand{\B}{{\mathcal{B}}}

\renewcommand{\H}{{\mathcal{H}}}

  \newcommand{\K}{{\mathcal{K}}}

\renewcommand{\O}{{\mathcal{O}}}

  \newcommand{\U}{{\mathcal{U}}}

\renewcommand{\phi}{\varphi}
\newcommand{\upchi}{{\raise.35ex\hbox{\ensuremath{\chi}}}}

\newcommand{\fA}{{\mathfrak{A}}}
\newcommand{\fB}{{\mathfrak{B}}}

\newcommand{\fI}{{\mathfrak{I}}}
\newcommand{\fJ}{{\mathfrak{J}}}

\newcommand{\fM}{{\mathfrak{M}}}

\newcommand{\fs}{{\mathfrak{s}}}
\newcommand{\ft}{{\mathfrak{t}}}

\newcommand{\bi}{{\mathbf{i}}}
\newcommand{\bj}{{\mathbf{j}}}

\newcommand{\bn}{{\mathbf{n}}}

\newcommand{\bx}{{\mathbf{x}}}
\newcommand{\by}{{\mathbf{y}}}
\newcommand{\bz}{{\mathbf{z}}}

\newcommand{\rC}{{\mathrm{C}}}

\newcommand{\qand}{\quad\text{and}\quad}
\newcommand{\qif}{\quad\text{if}\quad}
\newcommand{\qfor}{\quad\text{for}\quad}
\newcommand{\qforal}{\quad\text{for all}\quad}

\newcommand{\AND}{\text{ and }}
\newcommand{\FOR}{\text{ for }}

\newcommand{\diag}{\operatorname{diag}}

\newcommand{\id}{{\operatorname{id}}}

\newcommand{\ran}{\operatorname{Ran}}
\newcommand{\spn}{\operatorname{span}}

\newcommand{\bsl}{\setminus}
\newcommand{\ca}{\mathrm{C}^*}
\newcommand{\cenv}{\mathrm{C}^*_{\text{e}}}

\newcommand{\Fn}{\mathbb{F}_n^+}

\newcommand{\ol}{\overline}

\newcommand{\ltwo}{\ell^2}

\newcommand{\two}{\boldsymbol{2}}
\newcommand{\one}{\boldsymbol{1}}
\newcommand{\zero}{\boldsymbol{0}}

\begin{document}

\title[C*-envelopes]{C*-envelopes of tensor algebras\\ for multivariable dynamics}

\author[K.R.Davidson]{Kenneth R. Davidson}
\address{Pure Math.\ Dept.\\U. Waterloo\\Waterloo,
ON\; N2L--3G1\\CANADA}
\email{krdavids@uwaterloo.ca}

\author{Jean Roydor}
\address{Pure Math.\ Dept.\\U. Waterloo\\Waterloo,
ON\; N2L--3G1\\CANADA}
\email{jroydor@uwaterloo.ca}


\thanks{First author partially supported by an NSERC grant.}

\begin{abstract} We give a new very concrete description of the C*-envelope
of the tensor algebra associated to multivariable dynamical system. 
In the surjective case, this C*-envelope is described
as a crossed product by an endomorphism, and as a groupoid C*-algebra. 
In the non-surjective case, it is a full corner of a such an algebra.
We also show that when the space is compact,
then the C*-envelope is simple if and only if the system is minimal.
\end{abstract}

\date{}
\maketitle

\section{Introduction}

A multivariable dynamical system $(X,\sigma)$ is a locally compact Hausdorff space $X$
together with a family $\sigma=(\sigma_1,\dots,\sigma_n)$ of proper continuous maps from 
$X$ into itself.
In \cite{DK}, two natural universal operator algebras associated to this system were introduced.
The more tractable one is the tensor algebra $\A(X,\sigma)$, which is the universal operator
algebra generated by $\rC_0(X)$ and $n$ isometries $\fs_1,\dots,\fs_n$ with pairwise
orthogonal ranges satisfying the covariance relations
\[ f \fs_i = \fs_i (f \circ \sigma_i) \qforal f \in \rC_0(X) \AND 1 \le i \le n . \]
In that paper, there is a description of the C*-envelope of $\A(X,\sigma)$ as the 
Cuntz--Pimsner algebra of an associated C*-correspondence.
It also contains an explicit description of a norming family of boundary representations.
So in principle, a more explicit description of this C*-envelope should be available.
When $n=1$, Peters \cite{Pet} showed that the C*-envelope is a crossed product of
a related dynamical system constructed from the original by a projective limit construction.
In this paper, we show that a similar description is possible for $n\ge2$.
The C*-envelope is no longer a crossed product by an automorphism, but it is
a crossed product by an endomorphism.
It also is a groupoid C*-algebra of an related dynamical system.

The C*-envelope $\cenv(\A)$ of an operator algebra $\A$ is the unique minimal
C*-algebra (up to isomorphism fixing the image of $\A$) generated $j_0(\A)$,
where $j_0:\A\to\B(\H)$ is a completely isometric isomorphism.
This is characterized by the fact that if $j:\A\to\B(\H)$ is another completely isometric
isomorphism, then there is a surjective $*$-homomorphism 
$q:\ca(j(\A)) \twoheadrightarrow \cenv(\A)$ such that $qj=j_0$.
The existence of the C*-envelope was conjectured by Arveson \cite{Arv1}, and established
in many cases.  It was eventually proven by Hamana \cite{Ham}.
More recently, Dritschel and McCullough \cite{DM} provided a new proof.
The main ingredient was the notion of a maximal dilation.
A dilation $\pi$ of a representation $\rho$ is maximal if any further dilation of $\pi$
can only be accomplished by adding on a direct summand.
They prove that every maximal dilation factors through the C*-envelope.
In particular, if one starts with a completely isometric representation $\rho$
and constructs a maximal dilation $\pi$, then $\ca(\pi(\A))$ is the C*-envelope.

Arveson's definition of a boundary representation is equivalent to being maximal (in the
sense that it is a maximal dilation of itself) and extends to an irreducible $*$-representation
of the enveloping C*-algebra.
Dritschel and McCullough do not produce irreducible representations, but this is
not necessary to construct the C*-envelope.
However these irreducible representations are the analogue of the Choquet boundary, and
Arveson \cite{Arv_choq} shows that in the separable case, one can use a direct integral
decomposition to show that there are sufficiently many boundary representations to
construct the C*-envelope.
While a sufficient family of boundary representations for $\A(X,\sigma)$ were 
constructed in \cite{DK},
we find it convenient here to drop the irreducibility condition in order to have a larger family
of representations to work with to construct the C*-envelope more explicitly.

In the last section, we provide a direct proof that in the compact case,
simplicity of the C*-envelope is equivalent to minimality of the dynamical system.
Our proof is based on the representation of the C*-envelope as a crossed product
$\fB \times_\alpha \bN$ by an endomorphism.  The C*-algebra $\fB$ is an inductive
limit of homogeneous C*-algebras.  We provide an explicit description of the
$\alpha$-invariant and $\alpha$-bi-invariant ideals of $\fB$.
Then applying a result of Paschke \cite{Pa} will yield simplicity.
More general results of Schweizer \cite{Sch} for Cuntz--Pimsner algebras of
C*-correspondences could be used instead.

\section{Preliminaries} \label{S:prelim}

If $(X,\sigma)$ is a multivariable dynamical system, we form the (non-closed) covariance
algebra $\A_0(X,\sigma)$ as the space of polynomials in $n$ indeterminates
$\fs_1,\dots,\fs_n$ with coefficients in $\rC_0(X)$ where multiplication is
determined by the covariance relations $ f \fs_i = \fs_i (f \circ \sigma_i)$.
Let $\Fn$ be the free semigroup of all words in the alphabet $\{1,\dots,n\}$.
If $w = i_1i_1\dots i_k$, we write $\sigma_w$ for the map
$\sigma_{i_1} \circ\sigma_{i_2}\circ \dots \circ \sigma_{i_k}$;
and we write $\fs_w = \fs_{i_1} \dots \fs_{i_k}$.
Then a typical element of $\A_0(X,\sigma)$ is a finite sum $\sum_{w\in\Fn} \fs_w f_w$
where $f_w$ are arbitrary elements of $\rC_0(X)$.
The multiplication rule is just $(\fs_v f)(\fs_w g) = \fs_{vw} (f\circ\sigma_w)g$.

A row contractive representation $\rho$ of $\A_0(X,\sigma)$ is a homomorphism into $\B(\H)$
such that the restriction to $\rC_0(X)$ is a $*$-homomorphism
and $\big\| \big[ \rho(\fs_1) \ \dots \ \rho(\fs_n) \big] \big\| \le 1$.
The tensor algebra $\A(X,\sigma)$ is the universal operator algebra with these as
its completely contractive representations.
One can define the norm by taking a supremum over all such representations
into a fixed infinite dimensional Hilbert space sufficiently large to admit a
faithful representation of $\rC_0(X)$.
It is shown in \cite{DK} that every row contractive representation dilates to a row isometric
representation.
So it follows that $\fs_i$ are isometries with pairwise orthogonal ranges.

If $\rho$ is a (completely contractive) representation of an operator algebra $\A$
on a Hilbert space $\H$, we say that a representation $\pi$ of $\A$ on a Hilbert space
$\K$ containing $\H$ is a dilation of $\rho$ if $\K$ decomposes as
$\K = \H_- \oplus \H \oplus \H_+$ so that $\pi(A)$ is upper triangular with respect to this decomposition, and $\rho(A) = P_\H \pi(A) |_\H$ for all $A \in \A$.
As mentioned in the introduction, if the only dilations of $\rho$ have the form
$\pi = \rho \oplus \pi'$, then we say that $\rho$ is \textit{maximal}.

An easy way to obtain a representation of $\A(X,\sigma)$ is to pick a point $x \in X$
and define the \textit{orbit representation} $\lambda_x$.
This is defined on the Fock space $\ltwo(\Fn)$,  which has orthonormal basis
$\{ \xi_w : w \in \Fn\}$.
$\Fn$ acts on this space by the left regular action $L_v \xi_w = \xi_{vw}$.
Define
\[ \lambda_x(f) = \diag(f(\sigma_w(x))) \qand \lambda_x(\fs_i) = L_i \FOR 1 \le i \le n .\]
In general this is not maximal.
Indeed, this is maximal if and only if $x$ is not in the range of any map $\sigma_i$.
The representation $\Lambda_X = \bigoplus_{x\in X} \lambda_x$ is called the
\textit{full Fock representation}.  In \cite{DK}, it is shown that the full Fock representation
is completely isometric.

To obtain maximal representations, it is generally necessary to use an inductive limit construction.
An \textit{infinite tail representation} is given by an infinite sequence
$\bi = i_0 i_1 i_2 \dots$ in the alphabet $\{1,\dots,n\}$ and a
corresponding sequence of points $\bx = \{ x_s \in X: s \ge 0 \}$
such that $\sigma_{i_s}(x_{s+1}) = x_s$.
We will call such a pair $(\bi,\bx)$ an \textit{infinite tail} for $(X,\sigma)$.
For each $s\ge0$, let $\H_s$ denote a copy of Fock space
with basis $\{ \xi^s_w : w \in \Fn \}$.
Identify $\H_s$ with a subspace of $\H_{s+1}$ via $R_{i_s}$,
where $R_{i_s} \xi^s_w = \xi^{s+1}_{wi_s}$.
Consider the orbit representations $\lambda_{x_s}$ to be a representation on $\H_s$.
It is easy to see that $\lambda_{x_{s+1}} |_{\H_s} = \lambda_{x_s}$ for $s \ge 0$.
So we may define $\lambda_{\bi,\bx}$ to be the inductive limit of the
representations $\lambda_{x_s}$ on $\H = \ol{\bigcup_{s\ge0} \H_s}$.
This representation is always maximal.

In \cite{DK}, one required that the maximal representations also be irreducible.
This was accomplished by insisting that the orbits consist of distinct points.
For our purposes, it is convenient to ignore that requirement.
One still obtains a family of maximal representations.
Thus we have found two types of maximal representations.
 From these, we form two large representations of $\A(X,\sigma)$:
\begin{alignat*}{2}
 \lambda_{X,1} &:= \bigoplus{}_{x \in U} \ \lambda_x
 &\quad&\text{where } U = X \bsl \cup_{i=1}^{n} \sigma_i (X)\\
 \lambda_{X,2} &:= \bigoplus \lambda_{\bi,\bx} &\quad&\text{summing 
 over all possible infinite tails}\\
 \lambda_X &:= \lambda_{X,1} \oplus \lambda_{X,2} .
\end{alignat*}

We can state a result which follows from  \cite[Corollary~2.8]{DK}.
This is the case because $\Lambda_X$ is completely isometric, and
 $\lambda_X$ is a maximal dilation of $\Lambda_X$.

\begin{lemma}\label{L:maximal}
Let $(X,\sigma)$ be a multivariable dynamical system.
Then $\lambda_X$ is a completely isometric maximal representation of $\A(X,\sigma)$.
Consequently,
\[ \cenv(\A(X,\sigma)) = \ca(\lambda_X(\A(X,\sigma))). \]
\end{lemma}

\section{The surjective case}

In this section, we suppose that $(X,\sigma)$ is surjective in the sense that
$X = \bigcup_{i=1}^{n} \sigma_i (X)$.
When $n=1$, Peters \cite{Pet} used a projective limit construction on $(X,\sigma)$ to obtain
a new space $\tilde X$ and a homeomorphism $\tilde\sigma$, together with a projection
$p:\tilde X \to X$ such that $p \tilde\sigma = \sigma p$.
We first define an analogue of this construction.

Let $\bn = \{1,\dots,n\}$ with the discrete topology.
Set $Y=\bn^\bN \times X^\bN$ with the product topology.
Let $\tilde X$ be the subset of $Y$ consisting of all infinite tails for $(X,\sigma)$,
namely
\[
 \tilde{X}=  \{ (\bi,\bx) \in Y  : \sigma_{i_k}(x_{k+1}) = x_k \FOR k \ge 0 \} .
\]
The continuity of the maps ensures that $\tilde X$ is closed in $Y$.
If $\bi=(i_0,i_1,i_2,\dots)$, let $i\bi = (i,i_0,i_1,\dots)$.
Likewise, if $\bx=(x_0,x_1,x_2,\dots)$, let $(x,\bx) = (x,x_0,x_1,\dots)$
For $i=1,\dots,n$, we define $\tilde{\sigma}_i:\tilde{X} \to \tilde{X}$ by
\[
 \tilde{\sigma}_i  (\bi,\bx) = \big( i\bi, (\sigma_i(x_0),\bx) \big) =
 \big( (i,i_0,i_1,\dots),(\sigma_i(x_0),x_0,x_1,\dots) \big) .
\]
It is easy to see that $\tilde{\sigma}_i$ is a homeomorphism of $\tilde{X}$ onto $\tilde{X}_i$,
where
\[ \tilde{X}_i=  \{ (\bi,\bx) \in \tilde{X} : i_0=i  \}. \]
Observe that $\tilde{X}$ is the disjoint union of the sets $\tilde{X}_i$ for $1 \le i \le n$.
Thus we have constructed a new multivariable dynamical system $(\tilde{X},\tilde{\sigma})$
which we call the \textit{covering system} of $(X,\sigma)$.

Also define a projection $p:\tilde X \to X$ by $p (\bi,\bx) = x_0$.
Given any $x_0\in X$, surjectivity ensures that there is at least one choice of $i_0$
and $x_1\in X$ so that $\sigma_{i_0}(x_1) = x_0$.  Recursively, one can construct
a point $(\bi,\bx)\in\tilde X$ so that $p((\bi,\bx)) = x_0$.  So $p$ is surjective.
It is easy to see that $\sigma_i p = p \tilde\sigma_i$ for $1 \le i \le n$.

\begin{lemma} \label{L:p proper}
$\tilde X$ is (locally) compact when $X$ is (locally) compact.

The projection $p:\tilde X \to X$ is a proper map.

If $(\bi,\bx) \in \tilde X$, then there is a unique infinite tail in $\tilde X$ beginning at
this point.  Thus $\tilde{\tilde X} = \tilde X$.
\end{lemma}

\begin{proof}
If $X$ is compact, then $\tilde X$ is a closed subset of the compact space
$Y=\bn^\bN \times X^\bN$, and thus it is compact by Tychonoff's theorem.
However, if $X$ is not compact, then $Y$ will not be locally compact, and we need
to be more careful.
Let $X_\infty$ denote the one point compactification of $X$ obtained by adding
a point $\infty$.  Since each $\sigma_i$ is proper, it extends by continuity
to a map $\ol{\sigma}_i$ on $X_\infty$ by setting $\ol{\sigma}_i(\infty) = \infty$.
Form the compact space $\tilde X_\infty$.
Then the only points $(\bi,\bx)$ in $\tilde X_\infty$ for which any coordinate $x_j=\infty$
are the points $(\bi,\boldsymbol{\infty})$ where $\boldsymbol{\infty} = (\infty,\infty,\dots)$.
This is a compact set, and 
$\tilde X = \tilde X_\infty \setminus \bn^\bN \times \{\boldsymbol{\infty}\}$.
Since compact Hausdorff spaces are normal, it follows that $\tilde X$ is locally compact.

Let $K$ be a compact subset of $X$.
Since each $\sigma_i$ is proper, the sets $K_k = \bigcup_{|w|=k|} \sigma_w^{-1}(K)$
are compact for $k \ge 0$.
Therefore
\[p^{-1}(K) \subset \bn^\bN \times \prod_{k\ge0} K_k ,\]
which is compact.
Therefore $p$ is proper.

If $(\bi,\bx) \in \tilde X$, then the choice of the infinite tail beginning with this point
is uniquely determined.  This is because the maps $\tilde\sigma_i$ have disjoint ranges.
Indeed, the sequence of maps is just given by $\bi$ itself, and the points are
\[ \tilde x_k = \big( (i_k,i_{k+1}, \dots), (x_k, x_{k+1},\dots)) .\]
Therefore the covering space of $\tilde X$ is canonically
homeomorphic to $\tilde X$ itself via the projection map $\tilde p$.
\end{proof}

The new system has a number of advantages over the original.
In particular, it is possible to define an inverse map $\tau$ by
\[ \tau|_{\tilde X_i} = \tilde\sigma_i^{-1} \qfor 1 \le i \le n .\]
Clearly $\tau$ is everywhere defined on $\tilde X$ and is a local homeomorphism.
For $w \in \Fn$ with $|w|\ge1$, let $\tilde X_w = \tilde\sigma_w(\tilde X)$.
Observe that $\tilde X$ is the disjoint union of the clopen sets $\{\tilde X_w : |w|=k\}$ for
each $k\ge1$.
Let $\upchi_w$ denote the characteristic function of $\tilde X_w$.
This does not lie in $\rC_0(\tilde X)$ is $\tilde X$ (and hence $X$) is not compact.
But it does lie in the multiplier algebra.
Also let $p_k : \tilde X \to X$ be given by $p_k(\bi,\bx) = x_k$.
Observe that $p_k = p \circ \tau^k$.

\begin{lemma} \label{L:calculation}
Let the generators for $\A(\tilde X, \tilde\sigma)$ be $\ft_1,\dots,\ft_n$.
For any $f\in\rC_0(\tilde X)$ and $w\in\Fn$, $\ft_w f \ft_w^* = \upchi_w(f\circ\tau^{|w|})$.
\end{lemma}

\begin{proof}
Let $k = |w|$.
\begin{align*}
  \upchi_w(f\circ\tau^{|w|}) \ft_w
  &= \ft_w (\upchi_w \circ \tilde\sigma_w)(f \circ \tau^k \circ \tilde\sigma_w)\\
  &= \ft_w (1) (f \circ\id) = \ft_w f .
\end{align*}
Hence $ \ft_w f  \ft_w^* = \upchi_w(f\circ\tau^{|w|}) \ft_w  \ft_w^*$.
In the compact case, we can set $f=1$ and see that
$ \ft_w \ft_w^* = \upchi_w \ft_w  \ft_w^* \le \upchi_w$.
In general, this makes sense in the multiplier algebra.
Thus we have
\[ 1 = \sum_{|w|=k} \ft_w \ft_w^* \le \sum_{|w|=k} \upchi_w = 1 .\]
Hence $\ft_w \ft_w^* = \upchi_w$.
So the identity $\ft_w f  \ft_w^* = \upchi_w(f\circ\tau^{|w|})$ follows.
\end{proof}

When $n=1$, this consists of a single homeomorphism of $\tilde X$,
which allows Peters to construct a C*-crossed product.
When $n \ge2$, the situation is still much improved.
The first goal is to show that the C*-envelope is determined by this new system.

\begin{theorem}\label{T:projective}
Let $(X,\sigma)$ be a surjective multivariable dynamical system,
and let $(\tilde X,\tilde\sigma)$ be the associated covering system.
Let $\A(X,\sigma)$ and $\A(\tilde{X},\tilde{\sigma})$ be the associated tensor algebras.
Then
\begin{enumerate}[  $($\em i$)$]
\item $\A(X,\sigma)$ can be embedded into $\A(\tilde{X},\tilde{\sigma})$
via a completely isometric homomorphism.
\item $\cenv(\A(X,\sigma))=\cenv(\mathcal{A}(\tilde{X},\tilde{\sigma}))$.
\end{enumerate}
\end{theorem}

\begin{proof}
Let the generators of $\A(\tilde X,\tilde\sigma)$ be $\ft_1,\dots,\ft_n$.
Embed $C_0(X)$ into $C_0(\tilde{X})$ via $\rho(f) = f \circ p$.
This is a $*$-monomorphism because $p$ is surjective.
Then we can embed the covariance algebra $\A_0(X,\sigma)$ into $\A(\tilde X,\tilde\sigma)$
by defining $\rho(\fs_i f) = \ft_i \rho(f)$ and extending to the homomorphism
\[ \rho\big( \sum \fs_w f_w \big) = \sum \ft_w (f_w \circ p) .\]
The important observation is that $\rho$ satisfies the covariance relations
\[
 \rho(f) \rho(\fs_i ) = (f\circ p)  \ft_i = \ft_i (f \circ p \circ \tilde\sigma_i)
 = \ft_i (f \circ \sigma_i \circ p) = \rho(f) \rho(f \circ \sigma_i) .
\]
Therefore by the universal property of the tensor algebra, this extends to a
completely contractive representation of $\A(X,\sigma)$.

Let us verify that this map is a complete isometry.
We will use the fact that the full Fock representation is completely isometric.
For each $x\in X$, choose $(\bi,\bx) \in \tilde X$ so that $p((\bi,\bx)) = x$.
Observe that $\lambda_{(\bi,\bx)}\rho = \lambda_x$.  Indeed, both representations
send $\fs_i$ to the left shifts $L_i$, so it suffices to check what happens to $\rC_0(X)$.
If $w = j_k j_{k-1} \dots j_1$, then
\[ p \tilde\sigma_w(\bi,\bx) = \sigma_w p(\bi,\bx) = \sigma_w(x_0) .\]
So for any $f \in \rC_0(X)$,
\[
 \lambda_{(\bi,\bx)} \rho(f) \xi_w
 =  (f\circ p)(\tilde\sigma_w(\bi,\bx)) \xi_w
 = f(\sigma_w (x_0)) \xi_w = \lambda_x(f) \xi_w .
\]
Whence it follows for all elements $A \in \fM_m(\A_0(X,\sigma))$ that
\begin{align*}
 \|A\| &= \big\| (\id_{\fM_m} \otimes \Lambda_X) (A) \big\| \\&=
 \big\| (\id_{\fM_m} \otimes \Lambda_{\tilde X}) ( (\id_{\fM_m} \otimes\rho)(A))  \big\| \\&=
 \| (\id_{\fM_m} \otimes\rho)(A) \| .
\end{align*}
So this embedding is a complete isometry.

To prove (ii), we use the representations $\lambda_X$ and $\lambda_{\tilde X}$.
As these systems are surjective, we only need to consider infinite tail representations.
If $(\bi,\bx) \in \tilde X$, then the choice of the infinite tail beginning with this point
is uniquely determined by Lemma~\ref{L:p proper}.
This representation $\lambda_{\bi, (\tilde x_0,\tilde x_1,\dots)}$ will be denoted by
$\tau_{(\bi,\bx)}$.
Now if $x \in X$, the infinite tails beginning at $x$ are precisely the points
$(\bi,\bx) \in \tilde X$ such that $p((\bi,\bx)) = x$.
Arguing exactly as in the previous paragraph, we see that
$\tau_{(\bi,\bx)} \rho = \lambda_{(\bi,\bx)}$.

By Lemma~\ref{L:maximal}, we have
$\cenv(\A(\tilde X, \tilde\sigma)) = \ca(\lambda_{\tilde X}(\A(\tilde X, \tilde\sigma)))$.
Moreover the previous paragraph shows that $\lambda_{\tilde X} \rho \simeq \lambda_X$
yields a maximal completely isometric representation of $\A(X,\sigma)$ into this C*-algebra.
So $\cenv(\A(X,\sigma))$ is the C*-algebra generated by its image.
So it suffices to demonstrate that this is the whole algebra.

It is convenient to consider $\cenv(\A(\tilde X, \tilde\sigma))$ as generated by 
$\rC_0(\tilde X)$ and
$\{\ft_i f : 1 \le i \le n,\ f\in\rC_0(\tilde X) \}$.
Observe that $\ft_1,\dots,\ft_n$ are Cuntz isometries, meaning that they have 
orthogonal ranges with sum to the whole space.
This is evident in each infinite tail representation.
Moreover, one can see that $\ft_i\ft_i^* = \upchi_i$.
In the non-unital case, this makes sense in the multiplier algebra of 
$\cenv(\A(\tilde X, \tilde\sigma))$
which contains $\ft_1,\dots,\ft_n$ and all bounded continuous functions on $\tilde X$.
This follows from Lemma~\ref{L:calculation} because
\[ \ft_i\ft_i^* =  \ft_i 1 \ft_i^* = \upchi_i  \ft_i\ft_i^*  \upchi_i \le \upchi_i .\]
Since $1 = \sum  \ft_i\ft_i^* \le  \sum_i \upchi_i = 1$, we have equality.

The subalgebra $\rho(\cenv(\A(X,\sigma))$ is generated by the algebra of
functions $\{f \circ p : f \in \rC_0(X) \}$ and $\{\ft_i (f\circ p) : 1 \le i \le n,\ f\in\rC_0(X) \}$.
It suffices to show that all of $\rC_0(\tilde X)$ is in the smaller algebra.
To this end, it is enough to show that the algebra 
$\bigcup_{w \in \Fn} \ft_w \rho(\rC_0(X)) \ft_w^*$
is dense in $\rC_0(\tilde{X})$.
Now by Lemma~\ref{L:calculation}, for $|w|=k \ge 1$,
\[ \ft_w (f\circ p ) \ft_w^* = \upchi_w (f\circ p \circ \tau^k) \ft_w \ft_w^* =  \upchi_w (f\circ p_k) .\]
By the Stone--Weierstrass Theorem, it suffices to show that this subalgebra separates points
and does not vanish anywhere.  The latter is clear.
If $(\bi,\bx) \ne (\bj,\by)$, then either $\bi \ne \bj$ or for some $k\ge0$, $x_k \ne y_k$.
In the former case, there is a $k\ge 1$ so that the initial segment of $\bi$ is a word $w$
which differs from the initial segment of $\bj$.
So choose $f\in\rC_0(X)$ so that $f(x_k) \ne 0$, and if $y_k\ne x_k$, make $f(y_k) = 0$.
Then in either case,
\begin{align*}
 \ft_w (f\circ p)  \ft_w(\bi,\bx) &= \upchi_w(\bi,\bx) f(x_k) \ne 0,\\ \intertext{and}
 \ft_w (f\circ p)  \ft_w(\bj,\by) &= \upchi_w(\bj,\by) f(y_k) = 0.
\end{align*}

It follows that $\ca(\rho(\A(X,\sigma))) = \cenv(\A(\tilde X,\tilde\sigma))$.
\end{proof}

\begin{remark} 
This proof and the various algebraic relations imply that 
$\cenv(\A(X,\sigma))$ is the closed span of
\[ \{ \ft_v f \ft_w^* : v,w \in \Fn,\ f \in \rC_0(\tilde X) \} .\]
For instance, suppose $\tilde X$ compact. 
When 
$\spn\{  \ft_w \ft_w^* : w \in \Fn \}$ is dense in $\rC(\tilde X)$ 
(or equivalently when the characteristic functions of the sets $\tilde X_w$ 
for $w \in \Fn$ separate the points of $\tilde X$), then $\cenv(\A(X,\sigma))$ 
is isomorphic to the Cuntz algebra $\mathcal{O}_n$.
\end{remark}

\begin{example}\label{E:On}
Consider $X=\{1,\dots,n \}^{\mathbb{N}}$ and for $1 \leq i \leq n$, set
\[ \sigma_i((x_0,x_1,x_2,\dots))=(i,x_0,x_1,x_2,\dots) .\]
Obviously $X=\tilde X$ and the characteristic functions of the $\tilde X_w$'s 
separate the points of $\tilde X$.
So $\cenv(\A(X,\sigma))=\mathcal{O}_n$.
\end{example}

As a consequence of this theorem, we are able to describe the C*-algebra
$\cenv(\A(X,\sigma))$ as a groupoid C*-algebra.
Following \cite{AD} or \cite{Dea}, we denote by $\ca(\tilde{X},\tau)$ the groupoid 
C*-algebra associated to the local homeomorphism $\tau$.
The route to the proof is via Cuntz--Pimsner algebras of the associated 
C*-correspondences, which are shown to be isomorphic.

\begin{corollary}
Let $(X,\sigma)$ be a surjective multivariable dynamical system.
Then,
\[ \cenv(\A(X,\sigma)) \simeq \ca(\tilde{X},\tau) .\]
\end{corollary}

\begin{proof}
Deaconu, Kumjian and Muhly \cite{DKM} prove that $\ca(\tilde{X},\tau)$ is the
Cuntz-Pimsner C*-algebra associated to the C*-correspondence
\mbox{$E=C_0(\tilde{X})$} endowed with the $C_0(\tilde{X})$-valued inner product
\[
 \langle \xi,\eta \rangle (x) = \sum_{\tau(y)=x} \ol{\xi(y)}\eta(y)
 = \sum_{i=1}^n \ol{\xi(\tilde{\sigma}_i(x))}\eta(\tilde{\sigma}_i(x))
\]
for $\eta, \xi \in E$ and $x \in \tilde{X}$.
The left and right actions of $\rC_0(\tilde X)$ are given by
\[ f \cdot \xi (x)=f(x)\xi(x) \qand \xi \cdot f (x)=\xi(x)f(\tau(x)) \]
for $\xi \in E$ and $f \in C_0(\tilde{X})$.

On the other hand, in \cite{DK} it is shown that $\cenv(\A(\tilde{X},\tilde{\sigma}))$
is the Cuntz-Pimsner algebra associated to the C*-correspondence
$F=C_0(\tilde{X} \times n)$ over $C_0(\tilde{X})$ in the following way:
The $C_0(\tilde{X})$-valued inner product is
\[ \langle \xi,\eta \rangle (x)= \sum_{i=1}^n \overline{\xi(x,i)}\eta(x,i)
\]
and the left and right actions of $\rC_0(\tilde X)$ are
\[
 f \cdot \xi (x,i) = f(\tilde{\sigma}_i(x)) \xi(x,i) \qand
 \xi \cdot f (x,i)=\xi(x,i)f(x)
\]
for $\eta, \xi \in F$, $f \in C_0(\tilde{X})$ and $x \in \tilde{X}$.

To prove that these two Cuntz-Pimsner algebras are $*$-ismorphic, we will show that the
C*-correspondences $E$ and $F$ are unitarily equivalent,
i.e. there is a $C_0(\tilde{X})$-bimodule map from E onto F which preserves the inner products.

Define $h: \tilde{X} \times n \to \tilde{X}$ by  $h((x,i))=\tilde{\sigma}_i(x).$
Then consider the map $U$ from $\rC_0(\tilde{X})$ to $\rC_0(\tilde{X} \times n)$ by
$U\xi = \xi \circ h$.
It is easy to verify that $U$ is a $C_0(\tilde{X})$-bimodule map from E onto F
which preserves the inner products.
\end{proof}
\begin{example}
To illustrate this corollary, let's have another look at Example~\ref{E:On}.
In this example, the local homeomorphism $\tau: \tilde X \to \tilde X$ is just the left shift 
\[ \tau((x_0,x_1,x_2,\dots))=(x_1,x_2,\dots) .\]
By \cite[Example 1]{Dea}, the associated groupoid C*-algebra 
$\ca(\tilde{X},\tau)$ is $\mathcal{O}_n$. 
\end{example}

The next step is to describe the C*-envelope of $\A(X,\sigma)$ as the crossed product
$\fB \rtimes_\alpha \bN$ of a C*-algebra $\fB$ by a single endomorphism $\alpha$.
This construction was introduced by Cuntz \cite{C1} when he described his algebras
$\mathcal{O}_n$ as crossed products of UHF algebras by endomorphisms.
This construction applies more generally (see \cite{C2} and \cite{Sta}
for the non-unital case).

We recall how the crossed product by an endomorphism is defined.
Let $\fB$ be a C*-algebra and let $\alpha$ be an injective $*$-homomorphism of $\fB$ into itself.
In the unital case, there exists a unique C*-algebra $\fB \rtimes_{\alpha} \bN$
generated by $\fB$ and an isometry $S$ such that
\[ SbS^*=\alpha (b) \qforal b \in \fB \]
and satisfying the universal property:
for any  $*$-homomorphism $\pi$ of $\fB$ into $\B(\H)$ and any isometry $T \in B(\H)$
such that  $T\pi(b)T^*=\pi(\alpha (b))$, there is a $*$-homomorphism
$\tilde{\pi}: \fB \rtimes_{\alpha} \bN \to \B(\H)$ extending $\pi$
such that $\tilde{\pi}(S)=T$.
In the non-unital case, the isometry lives in the multiplier algebra.
$\fB \rtimes_{\alpha} \bN$ is defined as the universal algebra generated by
$\fB$ and $\{ Sb : b \in \fB \}$; and the universal property is that the
map $\pi$ above extends to $\tilde\pi$ satisfying $\tilde{\pi}(Sb)=T\pi(b)$
for all $b \in \fB$.

As usual, this crossed product has a family of gauge automorphisms $\gamma_z$
for $z \in \bT$ determined by
\[ \gamma_z(b) = b \FOR b\in\fB \qand \gamma_z(S)=zS .\]
A standard argument shows that  integration over $\bT$ yields a faithful conditional expectation
$\Gamma(A) = \frac1{2\pi}\int_0^{2\pi} \gamma_{e^{i\theta}}(A) \,d\theta$
from $\fB \rtimes_{\alpha} \bN$ onto $\fB$.

\smallbreak
We turn to the definition of $\fB$ in our setting.
We know from Theorem~\ref{T:projective} and the Cuntz and covariance relations
of $\fA := \cenv(\A(X,\sigma))$ that this algebra is the closed span of words of the form
\[ \{ \ft_v f \ft_w^* : v,w \in \Fn,\ f \in \rC_0(\tilde X) \} .\]
The universal property of the C*-envelope also guarantees that for each $z\in\bT$,
there are  $*$-automorphisms $\psi_z$ of $\fA$ determined by
$\psi_z(f)=f$ for $f \in \rC_0(\tilde X)$ and $\psi_z(\ft_i) = z\ft_i$ for $1\le i \le n$.
We define an expectation of $\Psi$ of $\fA$ into itself by integration:
$\Psi(A) = \frac1{2\pi}\int_0^{2\pi} \psi_{e^{i\theta}}(A) \,d\theta$.
It is easy to see that
\[
 \Psi(\ft_u f \ft_v^*) =
 \begin{cases} \ft_u f \ft_v^* &\qif |u|=|v| \\
                      0 &\quad\text{otherwise}
 \end{cases}
\]
Define
\[ \fB = \ran(\Psi) = \ol{\spn\{ \ft_u f \ft_v^* :
 u,v \in \Fn,\ |u|=|v|,\ f \in \rC_0(\tilde X) \}} .
\]

For $k\ge0$, define
\[
 \fB_k = \ol{\spn\{ \ft_u f \ft_v^* :
 u,v \in \Fn,\ |u|=|v|=k,\ f \in \rC_0(\tilde X) \}} .
\]
Since $\ft_v^*\ft_u = \delta_{u,v}$ when $|u|=|v|$, it is evident that $\fB_k$ is
a C*-subalgebra of $\cenv(\A(X,\sigma))$ which is $*$-isomorphic to
$\fM_{n^k}(C_0(\tilde{X}))$ via the map which sends
$\ft_u f \ft_v^*$ to $f \otimes E_{u,v}$, where
$\{E_{u,v} : |u|=|v|=k \}$ denote the matrix units of $\fM_{n^k}$.

Moreover, $\fB_k$ is contained in $\fB_{k+1}$ because if $|v|=|w|=k$, then
\[
 \ft_u f \ft_v^* = \ft_u f \sum_{i=1}^n \ft_i\ft_i^* \ft_v^* =
 \sum_{i=1}^n \ft_{ui} (f\circ\tilde\sigma_i) \ft_{vi}^*
\]
(Observe that this is not imbedded in the usual manner of UHF algebras because
of the fact that $\fB_0 = \rC_0(\tilde X)$ is imbedded into $\fB_k$ by
sending $f$ to the diagonal operator $\diag(f\circ\sigma_w)$.
Since the maps $\sigma_w$, for $|w|=k$, are homeomorphisms onto pairwise disjoint clopen
subsets of $\tilde X$, this carries $\fB_0$ onto the full diagonal of $\fB_k$.)
It follows that
\[ \fB = \ol{\bigcup\strut_{k\ge0} \fB_k} . \]
In particular, $\fB$ is a C*-subalgebra of $\fA$ which is the inductive
limit of homogeneous C*-algebras.

Define a proper isometry $V = \frac1{\sqrt n} \sum_{i=1}^n \ft_i$ in $\fA$.
Observe that
\[ \alpha(b) = VbV^* \qfor b \in \fB \]
determines a $*$-endomorphism.
Thus we can define the crossed product $\fB \times_\alpha \bN$.

\begin{theorem} \label{T:endo}
Let $(X,\sigma)$ be a multivariable dynamical system with $n \ge 2$. 
Then, with the above notation,
\[ \cenv(\A(X,\sigma)) \simeq \fB \times_\alpha \bN .\]
\end{theorem}

\begin{proof}
First we show that $\cenv(\A(X,\sigma))$ is generated by $\fB$ and $V$.
This is straightforward.  If $\ft_u f \ft_v^*$ is given, with $|v| \le |u|$,
let $k=|u|-|v|$.  Then
\[ (n^{k/2} \ft_u f \ft_v^* \ft_1^{*k} ) V^k = \ft_u f \ft_v^* .\]
So $\ft_u f \ft_v^*$ belongs to $\ca(\fB,V)$.
But these elements together with their adjoints span $\cenv(\A(X,\sigma))$.

By the universal property of $\fB \times_\alpha \bN$, there is a $*$-homomorphism
\[ \pi : \fB \rtimes_{\alpha} \bN \to \cenv(\A(X,\sigma)) \]
such that $\pi|_\fB = \id$ and $\pi(S) = V$.
This is surjective, and it is easy to see that
$\Psi \pi = \pi \Gamma$.
The gauge invariant uniqueness theorem shows that $\pi$ is an isomorphism.
\end{proof}

\begin{remark}\label{R:Peters result}
When $n=1$, $\fB$ is just $\rC_0(\tilde X)$ and the isometry $V$ is actually a unitary.
Thus this result recovers Peters result \cite{Pet} describing the C*-envelope 
as a C*-crossed product by $\mathbb{Z}$.
\end{remark}

\section{The non-surjective case: adding a tail}

The previous section only applies when the union of the ranges $\bigcup_{i=1}^n\sigma_i(X)$
is all of $X$.  When the system is not surjective, there is a technique called ``adding a tail''
which comes from the construction of graphs without sources from ones that have them.
This is now a standard procedure.

Given a dynamical system $(X,\sigma )$,
let $U=X \bsl \cup_{i=1}^{n} \sigma_i (X)$.
Define $T= \{ (u,k) : u \in \ol{U},\ k < 0 \}$ and $X^T=X \cup T$.
For each $1\le i \le n$, we extend $\sigma_i$ to a map $\sigma_i^T : X^T \to X^T$ by
\[ \sigma_i^T(u,k) = (u,k+1) \FOR k < -1, \qand \sigma_i^T(u,-1) = u .\]
We can consider the new multivariable dynamical system $(X^T,\sigma^T)$.

\begin{theorem} \label{T:surjective c.isom.}
Let $(X,\sigma)$ be a non-surjective multivariable dynamical system,
and let $(X^T,\sigma^T)$ be the system with an added tail. Then,
\begin{enumerate}[  \em (i)]
\item $\A(X,\sigma)$ can be embedded in $\A(X^T,\sigma^T)$ via
a completely isometric homomorphism and its image is completely contractively complemented in $\A(X^T,\sigma^T)$.
\item $\cenv(\A(X,\sigma))$ is a full corner of $\cenv(\A(X^T,\sigma^T))$.
\end{enumerate}
\end{theorem}

\begin{proof}
Let $\ft_1,\dots,\ft_n$ denote the generators of $\A(X^T,\sigma^T)$.
Extend each $f \in C_0(X)$ to a function $f^T \in C_0(X^T)$ by setting it
to be $0$  on $T$. Then we can embed $\A(X,\sigma)$ into $\A(X^T,\sigma^T)$ by
\[ j(f) = f^T \qand j(\fs_i f) = \ft_i f^T .\]
Clearly this is an algebra homomorphism.
Note that if $X$ is compact, then $j(\fs_i) = j(\fs_i 1) = \ft_i 1^T = \ft_i \upchi_X$
where $\upchi_X$ is the characteristic function of $X$ in $\rC(X^T)$.
Indeed, since $X$ is invariant for $\sigma^T$, we have
\[ j(\fs_w f) = \ft_w f^T = \upchi_X \ft_w f^T \upchi_X .\]

To see that this embedding is completely isometric, it suffices to consider the two
full Fock representations which are the direct sum of all orbit representations.
As $X$ is invariant under the maps $\sigma_i^T$, it is evident that orbit
representations $\pi_x$ and $\pi_x^T$ for the two systems coincide for all
$x \in X$.

If $(u,k)$ belongs to the tail $T$, then consider the representation $\pi_{(u,k)}^Tj$,
the restriction of the orbit representation $\pi_{(u,k)}^T$ to $\A(X,\sigma)$.
We claim that this is unitarily equivalent to $0^{(\alpha)} \oplus \pi_u^{(\beta)}$
where $0$ is the 1-dim\-ens\-ional zero representation, $\alpha = \sum_{s=0}^{k-1} n^s$
and $\beta = n^k$.
Indeed, on the basis vectors $\xi_w$ for $|w|<k$, $f^T(\sigma_w(u,k)) = f^T(u,k-|w|) = 0$.
Hence $\pi_{(u,k)}^T(j(A)) \xi_w = 0$ for all $A \in \A(X,\sigma)$ and all $|w|<k$.
Observe that for each word $w$ with $|w|=k$, the restriction of $\pi_{(u,k)}^T$ to
$\spn\{ \xi_{vw} : v \in \Fn \}$ is unitarily equivalent to $\pi_{u}^T$.  Hence the claim follows.

Combining these two observations, one sees that $\Pi_{X^T} j$ is completely isometric to
$\Pi_X$, and indeed they are the direct sum of the same representations with
different non-zero multiplicities.

The last assertion of (i) will be established in (ii) below.

To prove (ii), we use Lemma~\ref{L:maximal} to note that the representations
$\lambda_X$ and $\lambda_{X^T}$ are completely isometric maximal representations
of $\A(X,\sigma)$ and $\A(X^T,\sigma^T)$ respectively.
Note that any infinite tail $(\bi,\bx)$ of $(X,\sigma)$ is also an infinite tail of  $(X^T,\sigma^T)$.
So the sum of all of these representations, $\lambda_{X,2}$, is a direct summand of
$\lambda_{X^T}|_{\A(X,\sigma)}$.
The other summands of $\lambda_X$ are the orbit representations $\lambda_u$ for $u \in U$.
In $(X^T,\sigma^T)$, these can be dilated to infinite tail representations by setting
$\bx = (u,(u,-1),(u,-2),\dots)$ and taking an arbitrary infinite sequence $\bi$.
It is easy to see that these two options exhaust all of the summands of $\lambda_{X^T}$.
These latter infinite tail representations restrict to $\A(X,\sigma)$ to yield the representation
$0^{(\infty)} \oplus \lambda_u^{(\infty)}$.
The argument is essentially the same as the analysis above of the representations $\pi_{(u,k)}^T$,
except that the multiplicities are now countably infinite.
It follows that $\cenv(\A(X,\sigma))$ is isomorphic to the C*-subalgebra of $\cenv(\A(X^T,\sigma^T))$
generated by the image of $\A(X,\sigma)$.

We can summarize what we've proved so far in the following commutative diagram:
\[
\xymatrix @C+5pc {
 \A(X^T,\sigma^T) \ar@{^{(}->}[r]^{\lambda_{X^T}} &  \cenv(\A(X^T,\sigma^T)) \\
 \A(X,\sigma) \ar@{^{(}->}[r]^{\lambda_{X}} \ar@{^{(}->}[u]^{j}& \ar@{^{(}-->}[u]_{\tilde j}
 \cenv(\A(X,\sigma))
}
\]

More concretely,
\[
 \cenv(\A(X^T,\sigma^T)) =
 \ol{ \spn \{  \ft_u f_{u,v} \ft_u^* : f_{u,v} \in C_0(X^T),\ u,v \in \Fn \}}
\]
and
\[
 \cenv(\A(X,\sigma)) =
 \ol{\spn \{ \upchi_X \ft_u f_{u,v}^T \ft_v^*\upchi_X : f_{u,v} \in C_0(X),\ u,v \in \Fn \}}.
\]
Thus it is evident that
\[ \cenv(\A(X,\sigma)) = \upchi_X  \cenv(\A(X^T,\sigma^T)) \upchi_X .\]
So this is a corner of $\cenv(\A(X^T,\sigma^T))$.

To see that this is a full corner, observe that if $|w|=k$, then
\[ (\ft_u f \ft_w^*) \upchi_X (\ft_w g \ft_v) = \ft_u (fg \upchi_{\tau^k(X)}) \ft_v .\]
Since the sets $\tau^k(X)$ are an increasing sequence of clopen sets with union $X^T$,
$\upchi_{\tau^k(X)}$ is an approximate unit for $\rC_0(X^T)$.
It follows that $\ft_u fg \ft_v$ lies in
$\cenv(\A(X^T,\sigma^T)) \upchi_X \cenv(\A(X^T,\sigma^T))$.
Thus it is a full corner.

Moreover we see that the map taking $A \in \A(X^T,\sigma^T)$ to $A \upchi_X$ is a completely
contractive idempotent projection onto $\A(X,\sigma)$ with complementary map sending
$A$ to $A \upchi_T$, which is also a complete contraction.
Thus $\A(X,\sigma)$ is  completely contractively complemented in $\A(X^T,\sigma^T)$.
\end{proof}

\section{Simplicity}

In this last section, we consider when the C*-envelope of $\A(X,\sigma)$ is simple.
We will soon restrict our attention to the compact case.

\begin{definition}
A subset $A \subset X$ is \textit{invariant} for $(X,\sigma)$ if $\sigma_i(A) \subset A$ for $1 \le i \le n$.
Say that $A$ is \textit{bi-invariant} if in addition, $\sigma_i^{-1}(A) \subset A$ for $1 \le i \le n$.

If $(X,\sigma)$ is a dynamical system, let the orbit of a point $x$ be
$\O^+(x) = \{ \sigma_w(x) : x \in \Fn \}$ and let the \textit{full orbit} of $x$ be
the smallest set $\O(x)$ containing $x$ which is \textit{bi-invariant}.

If $(X,\sigma)$ is a compact dynamical system,
we say $(X,\sigma)$ is \textit{minimal} if and only if every orbit is dense in $X$
or equivalently, there are no proper closed invariant sets.
\end{definition}

\begin{example}\label{E:no bi-invariant set}
Consider the space $X = \{ 0,1,2 \}$.  For $i=1,2$, define a map
\[
 \sigma_i(j) =
 \begin{cases}  i  &\qif j=i\\ 0  &\quad\text{otherwise}
 \end{cases}
\]
Then $\{0\}$ is a proper closed invariant set.
However, one can easily check that X has no proper bi-invariant sets.

Observe that $\tilde X$ consists of the points
\begin{gather*}
 (\one,\one), (\two,\two), (\bi,\zero) \qfor \bi\in\bn^\bN \qand \\
 \big((i_1\dots i_{k-1}2\one),0^k\one) \big),  \big((i_1\dots i_{k-1}1\two),0^k\two) \big)
 \qfor k \ge 1,
\end{gather*}
where $\zero$, $\one$ and $\two$ represent an infinite string of the digit 0,1 or 2
respectively, and $0^k$ represents a string of $k$ zeros.
Now
\[ p^{-1}(0) = \tilde X \setminus\{(\one,\one), (\two,\two)\} \]
is invariant, but not bi-invariant.  It contains the subset
$\{(\bi,\zero) : \bi\in\bn^\bN\}$ which is a proper closed bi-invariant set.

The difference in the two situations results from the fact that the inverse map
$\tau$ on $\tilde X$ takes a point to its \textit{unique} preimage under the
maps $\tilde\sigma_i$.  In $X$, the point $0$ has multiple preimages.
\end{example}

\begin{example}\label{E:noncompact}
Consider the space $X = \bN$ with the map $\sigma(n) = n+1$.
Then $X$ contains many proper closed invariant sets, but has no proper
bi-invariant set.  In this case, $X=\tilde X$.  So no improvement is obtained.
This system is not surjective.  Adding a tail yields the analogous system on
$\bZ$, which also has many invariant sets, but no proper bi-invariant sets.
The algebra $\cenv(\A(X,\sigma))$ is the compact operators, which is simple.
\end{example}

\begin{proposition} \label{P:minimal lift}
Let $(X,\sigma)$ be a compact dynamical system.
Minimality implies that $(X,\sigma)$ is surjective; and the following are equivalent:
\begin{enumerate}[  $($\em 1$)$]
\item $(X,\sigma)$ is minimal.
\item $(\tilde X,\tilde\sigma)$ is minimal.
\item $(\tilde X,\tilde\sigma)$ has no proper closed bi-invariant subset.
\end{enumerate}
\end{proposition}

\begin{proof}
If $(X,\sigma)$ is not surjective, then $A = \bigcup_{i=1}^n \sigma_i(X)$
is a proper closed invariant subset.  So minimality implies surjectivity.

If $A\subset X$ is a proper closed invariant set,
then $\tilde A = p^{-1}(A)$ is a proper closed invariant subset
of $(\tilde X,\tilde\sigma)$.  So (2) implies (1).

Suppose that $B \subset \tilde X$ is closed and invariant.
Define a sequence of subsets
\[ B_0=B \AND B_{k+1} = \bigcup_{i=1}^n \tilde\sigma_i(B_k) \qfor k \ge 0 .\]
Then this is a decreasing sequence of non-empty compact invariant sets.
Hence $B_\infty = \bigcap_{k\ge0} B_k$ is a non-empty closed invariant set.

We claim that $B_\infty$ is bi-invariant.
Indeed, if $x\in B_\infty$, then $x = \tilde\sigma_i(y)$
for a unique choice of $i$ and $y$, namely $y=\tau(x)$ and $i$ is
determined by membership in $\tilde X_i$, which are disjoint sets.
Since $x \in B_{k+1}$, it follows that $y\in B_k$.
This holds for all $k\ge0$, and thus $y \in B_\infty$.
So (3) implies (2).

Clearly, (2) implies (3).
Finally suppose that $(X,\sigma)$ is minimal, and fix a point $(\bi,\bx) \in \tilde X$.
We will show that the orbit $\O^+((\bi,\bx))$ is dense in $\tilde X$.
To this end, suppose that $(\bj,\by)$ is an arbitrary point in $\tilde X$.
A basic open neighbourhood of $(\bj,\by)$ is given by an integer $p$
and open sets $U_k$ of $y_k$ for $0 \le k \le p$:
\[  U = \{(\bi,\bz) : i_k = j_k \FOR 0 \le k < p \AND z_k \in U_k \FOR 0 \le k \le p \}  . \]
Moreover, we may suppose that $\sigma_{j_k}(U_{k+1}) \subset U_k$ by
replacing $U_1$ by a smaller open neighbourhood of $y_1$ which is mapped
by $\sigma_{j_0}$ into $U_0$ by using the continuity of $\sigma_{j_0}$.
Then replace $U_2$ by a smaller open set, etc.

Since $X$ is minimal, $\O^+(x_0)$ is dense in $X$.
Select a word $w$ so that $\sigma_w(x_0) = z_p \in U_p$.
Define $z_k = \sigma_{j_k}(z_{k+1})$ for $k=p-1,\dots,0$.
Consider the point
\[
 \tilde\sigma_{j_0j_1\dots j_{p-1}w}((\bi,\bx))  =
\big( (j_0j_1\dots j_{p-1}w\bi), (z_0,z_1,\dots,z_p,\dots) \big).
\]
This evidently belongs to $U$.  Hence $\O^+((\bi,\bx))$ is dense in $\tilde X$.
So $(\tilde X,\tilde\sigma)$ is minimal.
\end{proof}

Recall that in the surjective case, we have expressed $\cenv(\A(X,\sigma))$ as the
crossed product of a C*-algebra $\fB$ by an endomorphism, $\fB \times_\alpha \bN$, where
$\alpha(b) = \frac1 n \sum_{i,j=1}^n \ft_i b \ft_j^*$ and $\fB$ is the span of all elements
$\ft_u f \ft_v^*$ for $|u|=|v|$ and $f \in \rC_0(\tilde X)$.
An ideal of $\fB$ will intersect $\rC_0(\tilde X)$ in an ideal, and so will have the
form $I_F = \{ f \in \rC_0(\tilde X) : f|F = 0 \}$, where $F$ is a closed subset of $\tilde X$.

\begin{definition}
Call a subset $F$ of $(\tilde X,\tilde\sigma)$ \textit{robust} if
$F$ contains $\tilde\sigma_w \tau^{|w|}(x)$
for every $x \in F$ and $w \in \bF_n$.
\end{definition}

The robust closed subsets are not necessarily bi-invariant, but they they do have
the property that if $|v|=|w|$ and $y\in \tilde X$ such that $\tilde\sigma_v(y) \in F$,
then $\tilde\sigma_w(y) \in F$ as well. This is the key concept in the following result.

\begin{lemma} \label{L:alpha invariant}
Let $(X,\sigma)$ be a surjective, multivariable dynamical system.
Then the $\alpha$-invariant ideals of $\fB$ are in bijective correspondence with
the closed robust $\tau$-invariant subsets of $\tilde X$ via the map taking $\fJ$ to
$F$, where $\fJ \cap \rC_0(\tilde X) = I_F$.
\end{lemma}

\begin{proof}
Let $\fJ$ be an $\alpha$-invariant ideal  of  $\fB$.
Now $\fB$ is the inductive limit of subalgebras $\fB_k \simeq \fM_{n^k}(\rC_0(\tilde X))$.
Hence $\fJ$ is the closed union of the ideals $\fJ_k := \fJ \cap \fB_k$ for $k\ge 0$.
These ideals have the form $\fJ_k \simeq \fM_{n^k}(I_{F_k})$ for some closed
subsets $F_k$ of $\tilde X$.

The injection of $\fJ_k$ into $\fJ_{k+1}$ sends $\ft_u f \ft_v$, where $|u|=|v|=k$, to
$\sum_{i=1}^n \ft_{ui}(f\circ\tilde\sigma_i)\ft_{vi}$.
This shows that if $f\in I_{F_k}$, then $f\circ \tilde\sigma_i$ belong to $I_{F_{k+1}}$.
This implies that
\[
 \bigcup_{i=1}^n \tilde\sigma_i(F_{k+1}) \subset F_k.
\]
In particular, this implies that $F_{k+1} \subset \tau(F_k)$.

On the other hand, if $f \in I_{F_{k+1}}$ and $|u|=|v|=k$,
then $\fJ_{k+1}$ contains $\ft_u(\ft_i f \ft_i^*) \ft_v^* = \ft_u \upchi_i(f\circ\tau) \ft_v^*$
by Lemma~\ref{L:calculation}.
Hence $\upchi_i(f\circ\tau)$ belongs to $I_{F_k}$, and thus vanishes on $F_k$.
This implies that
\[ \tau(F_k) = \bigcup_{i=1}^n \sigma_i^{-1}(F_k) \subset F_{k+1} .\]
Together these relations show that
\[ F_{k+1} = \tau(F_k) \qand F_k = \bigcup_{i=1}^n \tilde\sigma_i(F_{k+1}) .\]
Therefore $F_k = \tau^k(F_0)$ and $F_0 = \bigcup_{|w|=k} \tilde\sigma_w(F_k)$.
Hence $F_0$ is robust.

Since $\alpha(\fJ) \subset \fJ$, we see that if $f\in I_{F_k}$ and $|u|=|v|=k$, then
\[ \alpha( \ft_u f \ft_v^*) = \frac1n \sum_{i,j=1}^n \ft_i \ft_u f \ft_v^* \ft_j^* \in \fJ_{k+1} .\]
Thus $f \in I_{F_{k+1}}$.
Therefore $I_{F_k} \subset  I_{F_{k+1}}$, whence $F_{k+1}\subset F_k$.
It follows that $\tau(F_k) \subset F_k$.
So each $F_k$ is $\tau$ invariant.
In particular, $F_0$ is robust and $\tau$-invariant.

Conversely, suppose that  $F_0$ is robust and $\tau$-invariant.
Define $F_k = \tau^k(F_0)$ for $k \ge 1$.
It follows from the robustness of $F_0$ that each $F_k$ is
also robust.  In particular,  $F_k = \bigcup_{i=1}^n \tilde\sigma_i(F_{k+1})$.
Let
\[
 \fJ_k = \spn\{ \ft_u f \ft_v^* : |u|=|v|=k \AND f \in I_{F_k} \}
 \qand \fJ = \ol{ \bigcup_{k\ge0} \fJ_k} .
\]
Reversing the arguments above shows that $\fJ_k \subset \fJ_{k+1}$ and
$\alpha(\fJ_k) \subset \fJ_{k+1}$ for all $k \ge 0$.
It follows that the union $\fJ$ is an $\alpha$-invariant ideal of $\fB$.

An element $f\in \fJ_k \cap \rC_0(\tilde X)$ is represented in $\fB_k$ as
$\sum_{|u|=k} \ft_u (f\circ \tilde\sigma_u) \ft_u^*$.
This requires $f\circ \tilde\sigma_u \in I_{F_k}$ for all $|u|=k$.
Arguing as above, this means that $f\in I_{F_0}$.
Therefore $\fJ_k \cap \rC_0(\tilde X) = I_{F_0}$ for all $k\ge0$,
and thus the same holds for $\fJ$.
This shows that the map taking $\fJ$ to $F_0$ is a surjection onto the
collection of $\tau$-invariant, closed robust sets.
Also the details of the structure show that $F_0$ determines the
sets $F_k$ and hence the ideals $\fJ_k$.  So this map is injective.
\end{proof}

\begin{remark}\label{R:alpha invariant}
There are $\alpha$-invariant ideals determined by non-constant sequences of sets.
An easy example starts with $X = [0,1] \times 2^\bN$ and two maps
$\sigma_i(x,\bi) = (x^2,i\bi)$.  Since $\sigma_1$ and $\sigma_2$ are injective
maps with complementary ranges, one sees that $\tilde X = X$.
Consider $F_0 = [r,1]\times 2^\bN$ for any $0 < r < 1$.  Then
$F_k = [r^{2^{-k}},1]\times 2^\bN$ satisfy the relations, and therefore
determine a proper $\alpha$-invariant ideal of $\fB$.
The only proper closed bi-invariant sets are $\{0\}\times 2^\bN$ and $\{1\}\times 2^\bN$.
\end{remark}

We will be interested in $\alpha$-invariant ideals of $\fB$ which are obtained
from ideals of $\fB \times_\alpha \bN$ by intersection with $\fB$.
This puts an additional constraint on $\fJ$, namely that $\ft_i^* \fJ \ft_j \subset \fJ$.
Reasoning as above, one sees that if $f \in I_{F_{k+1}}$, then $f \in I_{F_k}$.
So we deduce that $F_k=F_0$ for all $k$, and $F$ is a bi-invariant set.
So the ideals associated to bi-invariant sets play a more important role for us.
An apparently weaker but more intrinsic condition leads to the same result.

\begin{definition}
An ideal $\fJ$ of $\fB$ is $\alpha$-bi-invariant if $\alpha(\fJ) \subset \fJ$
and whenever $\alpha(b) \in \fJ$, then $b \in \fJ$.
\end{definition}

\begin{corollary} \label{C:alpha bi-invariant}
Let $(X,\sigma)$ be a surjective multivariable dynamical system.
Then the $\alpha$-bi-invariant ideals of $\fB$ are in bijective
correspondence with the closed $\tilde\sigma$-bi-invariant subsets of $\tilde X$
via the map sending $\fJ$ to $\fJ \cap \rC_0(\tilde X) = I_F$.
\end{corollary}

\begin{proof}
In particular, $\fJ$ is $\tau$-invariant.
So we adopt the notation of the previous proof to describe $\fJ$.
Suppose that $f \in I_{F_k}$.
Then $\tau^k(f) = n^{-k} \sum_{|u|=|v|=k} \ft_u f \ft_v^*$ belongs to $\fJ_k$.
By the bi-invariance, we have $f\in \fJ$.
Therefore $I_{F_0}=I_{F_k}$.  Thus
\[ F_0 = \tau(F_0) = \bigcup_{i=1}^n \tilde\sigma_i(F_0) .\]
In other words, $F_0$ is bi-invariant.

Conversely, if $F_0$ is bi-invariant, then following the construction of
Lemma~\ref{L:alpha invariant}, we have $F_k=F_0$ and
and $F_k = \bigcup_{i=1}^n \tilde\sigma_i(F_{k+1})$ for all $k\ge 0$.
It is easy to see that if $b\in \fJ_{k+1}$, then $\ft_i^* b \ft_j$ belongs to $\fJ_k$
for all $i,j$.  In particular, if $\alpha(b) \in \fJ_{k+1}$, then $b \in \fJ_k$.
So $\fJ$ is $\alpha$-bi-invariant.
\end{proof}

These results are more transparent in the unital case.
We have the following variant on Proposition~\ref{P:minimal lift}.

\begin{corollary} \label{C:alpha invariant}
Let $(X,\sigma)$ be a surjective, compact, multivariable dynamical system
with $n\ge2$.
Then the following are equivalent:
\begin{enumerate}[  $($\em 1$)$]
\item $X$ is minimal.
\item $\fB$ has no proper $\alpha$-invariant ideals.
\item $\fB$ has no proper $\alpha$-bi-invariant ideals.
\end{enumerate}
\end{corollary}

\begin{proof}
By Proposition~\ref{P:minimal lift}, minimality of $X$ is equivalent to having
no proper closed bi-invariant subsets in $(\tilde X,\tilde\sigma)$.
So by Corollary~\ref{C:alpha bi-invariant}, this is equivalent to having
no proper $\alpha$-bi-invariant ideals in $\fB$.
So (1) and (3) are equivalent.

Clearly, (3) implies (2).  However, by Lemma~\ref{L:alpha invariant}, (2) implies
the existence of a proper robust $\tau$-invariant subset $F$.
It follows that $\bigcap_{k\ge0} \tau^k(F)$ is a proper closed bi-invariant set.
\end{proof}

We can now prove the main result of this section.

\begin{theorem}
Let $(X,\sigma)$ be a compact multivariable dynamical system ($n \geq 2$). Then
$\cenv(\mathcal{A}(X,\sigma))$ is simple if and only if $(X,\sigma)$ is minimal.
\end{theorem}

\begin{proof}
First suppose that $X$ is surjective and is not minimal.
By Proposition~\ref{P:minimal lift}, this is equivalent to the existence
of a proper closed bi-invariant subset $F$ of $\tilde X$.
Corollary~\ref{C:alpha invariant} provides
an $\alpha$-bi-invariant ideal of $\fB$ determined by $F$.
Define
\[ \fI = \spn\{ \ft_u f \ft_v^* : f \in I_F \AND u,v \in \Fn \}. \]
Since $F$ is $\tilde\sigma$-invariant, this is seen to be an ideal of $\cenv(\A(X,\sigma))$.
We use the $\tau$-invariance of $F$ to see that $\fI \cap \rC_0(\tilde X) = I_F$.
This follows from Lemma~\ref{L:calculation} since this intersection will contain
$\ft_w f \ft_w^* = \upchi_w(f \circ\tau^{|w|})$.
It follows that $\cenv(\A(X,\sigma)$ is not simple.

When $(X,\sigma)$ is not surjective, it is definitely not minimal.
Moreover, it is clear that the surjective system $(X^T,\sigma^T)$ is not minimal either.
Thus $\cenv(\A(X^T,\sigma^T))$ is not simple.
Now $\cenv(\A(X,\sigma))$ is a full corner of  $\cenv(\A(X^T,\sigma^T))$,
so they are Morita equivalent.  In particular, there is a bijective correspondence
between their ideals.  So  $\cenv(\A(X,\sigma))$ is not simple either.

Now suppose that $X$ is minimal.
Then in particular $(X,\sigma)$ is surjective.
Lemma~\ref{L:alpha invariant} shows that $\fB$ has no proper $\alpha$-invariant ideals.
We will apply a result of Paschke \cite{Pa} to see that
$\cenv(\A(X,\sigma)) \simeq \fB \times_\alpha \bN$ is simple.
The C*-algebra $\fB$ is the inductive limit of matrix algebras over $\rC(\tilde X)$.
Since $X$ is compact, these algebras are unital and the imbeddings are unital.
It follows from Johnson \cite{John} that $\fB$ is strongly amenable.
Since $\fB$ has no proper $\alpha$-invariant ideals, Paschke's result shows
that $\fB \times_\alpha \bN$ is simple.
\end{proof}

\begin{remark}\label{R}
When $n=1$, the previous theorem is false. 
For example, when $X$ is just one point and $\sigma$ is the identity mapping, 
then $\A(X,\sigma)$ is the disc algebra. 
This system is obviously minimal, but 
$\cenv(\A(X,\sigma))=\rC(\bT)$ is not simple. 
Since $\tilde X = X$ is a point, this follows, for example, 
from Peters \cite{Pet} identification
\[ \cenv(\mathcal{A}(X,\sigma))  = \rC(\tilde X) \times_\alpha \bZ = \rC(\bT) .\]
It is well-known that one must require $\tilde X$ to be infinite
to apply the previous simplicity criteria. 
But when $n \geq 2$, $\tilde X$ is necessarily infinite.
\end{remark} 

The results of Schweizer \cite{Sch} show that unital Cuntz--Pimsner algebras
are simple if and only if the analogue of the C*-subalgebra $\fB$ has no invariant ideals.
This result could be applied here, but is technically more difficult.

Naturally one wants a nice condition that is equivalent to simplicity in the
non-compact case as well.  We suspect that this should hold precisely when $\tilde X$
has no proper bi-invariant closed subsets.  Example~\ref{E:noncompact} shows that
this is not equivalent to proper invariant sets in $X$.  Example~\ref{E:no bi-invariant set}
shows that this is not equivalent to proper bi-invariant sets in $X$ even when $X$ is compact.
So we have no good idea about a dynamical condition on $X$ which determines this property.



\end{document}